\documentclass{amsart}
\usepackage{amssymb}
\usepackage{amsmath}
\usepackage{stmaryrd}
\usepackage{color}
\usepackage{comment}
\usepackage{enumitem}
\usepackage[all]{xy}

\newtheorem{prop}{Proposition}[section]
\newtheorem{lem}[prop]{Lemma}
\newtheorem{thm}[prop]{Theorem}
\newtheorem{cor}[prop]{Corollary}

\theoremstyle{remark}
\newtheorem{rem}[prop]{Remark}

\newtheorem{ex}[prop]{Example}

\theoremstyle{definition}
\newtheorem{defi}[prop]{Definition}

\newcommand{\mz}{{\mathbb Z}}

\usepackage{scalerel}

\DeclareMathOperator{\Ch}{Ch}

\DeclareMathOperator{\M}{Mot}
\DeclareMathOperator{\End}{End}
\DeclareMathOperator{\Hom}{Hom}

\DeclareMathOperator{\mult}{mult}
\DeclareMathOperator{\rk}{rk}
\DeclareMathOperator{\SB}{SB}
\DeclareMathOperator{\Spec}{Spec}

\DeclareMathOperator{\Tr}{Tr}

\newcommand{\RM}[2]{\mathrm{Mot}^{\mathrm{KS}}_{#1}({#2})}
\newcommand{\IM}[2]{\mathrm{Mot}^{\mathrm{I}}_{#1}({#2})}

\newcommand{\PGO}{\mathrm{PGO}}

\DeclareMathOperator{\Gal}{Gal}

\newcommand{\CH}{\mathrm{CH}}

\title[]{Higher Tate traces of Chow motives}
\author{Charles De Clercq}
\address{Université Sorbonne Paris Nord, Intitut Galilée, Laboratoire Analyse, Géométrie et Applications, Villetaneuse, France.}
\email{declercq@math.univ-paris13.fr}
\author{Anne Qu\'eguiner-Mathieu}
\address{Université Sorbonne Paris Nord, Intitut Galilée, Laboratoire Analyse, Géométrie et Applications, Villetaneuse, France.}
\email{queguin@math.univ-paris13.fr}
\begin{document}

\begin{abstract}We establish the complete classification of Chow motives of projective homogeneous varieties for $p$-inner semi-simple algebraic groups, with coefficients in $\mz/p\mz$. Our results involve a new motivic invariant, the Tate trace of a motive, defined as a pure Tate summand of maximal rank. They apply more generally to objects of the Tate subcategory generated by upper motives of irreducible, geometrically split varieties satisfying the nilpotence principle. Using Chernousov-Gille-Merkurjev decompositions and their interpretation through Bialynicki-Birula-Hesselink-Iversen filtrations due to Brosnan, we then generalize the characterization of the motivic equivalence of inner semi-simple groups through the higher Tits $p$-indexes. We also define the motivic splitting pattern and the motivic splitting towers of a summand of the motive of a projective homogeneous variety, which correspond for quadrics to the classical splitting pattern and Knebusch tower of the underlying quadratic form. 
\end{abstract}

\maketitle

\section{Introduction} 

Chow motives provide a powerful tool to study the geometry of varieties without rational points, including projective homogeneous varieties, such as quadrics and Severi-Brauer varieties. In this article we introduce a new family of motivic invariants, the Tate traces, which measures the isotropy level of motives. More precisely, the Tate trace of a motive is a pure Tate direct summand of maximal rank. 

The classification of projective homogeneous varieties up to motivic isomorphism is a long-standing problem, opened more than 30 years ago with the classification over separably closed fields, involving Poincaré polynomials \cite{Kock}. In this work, through the Tate traces, we obtain the classification  over an arbitrary field for the motives with coefficients in $\mz/p\mz$ of projective homogeneous varieties associated to $p$-inner semi-simple algebraic groups, see Theorem~\ref{main.thm}. 
For arbitrary semi-simple groups, the statement is not valid anymore, as explained in Remark~\ref{outer.rem}. 

The theory of motives of projective homogeneous varieties was brought to the forefront of modern algebra in the nineties, with Voevodsky's proof of Milnor's conjecture \cite{voe} which relied on the Rost motive. This inspired an extensive study of the motives of projective quadrics and projective homogeneous varieties by Vishik, Karpenko, Merkurjev and others, leading notably to numerous breakthroughs for classical problems in the algebraic theory of quadratic forms \cite{K-Hoff,Kar-Mer1,vish-lens,vish-uinv}. Among these results is the classification of motives of projective quadrics \cite[Thm. 4.18]{vish-lens} (see \cite[Thm. 93.1]{EKM} for the result in arbitrary characteristic).

\begin{thm}[Vishik's criterion]Let $Q$ and $Q'$ be smooth projective quadrics, given by quadratic forms $q$ and $q'$ of the same dimension over a field $F$.
The motives of $Q$ and $Q'$ (with coefficients in $\mz$ or $\mz/2\mz$) are isomorphic if and only if for all field extensions $L/F$, the Witt indexes of $q_L$ and $q'_L$ are equal.
\end{thm}

Inspired by one of the main tools used by Vishik, the \v{C}ech simplicial schemes, Karpenko introduced in \cite{K-inner} their avatars for projective homogeneous varieties, the upper motives. Upper motives are easier to handle than motives of varieties, yet they encode deep geometric information, such as their canonical $p$-dimension \cite{K-ICM}. They were used by Karpenko to classify motives of Severi-Brauer varieties (see \cite{dC,KZ} for other applications). The classification of motives for more general flag varieties of inner type $A_n$ is given by \cite[Thm 19]{dC}.

 Criteria for motivic isomorphisms were achieved in several other cases such as Borel varieties \cite{PSZ}, involution varieties \cite{dCQZ} and some exceptional projective homogeneous varieties \cite{GPS,NSZ,PS}. However, all the above cases deal with projective homogeneous varieties of the same type, and with respect to semi-simple groups with the same Dynkin diagram. As opposed to this, the criterion given in theorem \ref{main.thm} compares the motives with $\mz/p\mz$ coefficients of two arbitrary projective homogeneous varieties for two $p$-inner semi-simple algebraic groups of arbitrary and possibly different types. Moreover, the proof does not rely on the classification of semi-simple groups. The class of $p$-inner semi-simple groups is quite large: all absolutely simple groups are $p$-inner for some prime $p$, except the trialitarian groups of type $^6\!D_4$. Products of $p$-inner semi-simple groups are $p$-inner, as well as Weyl restrictions of such, with respect to separable extensions of degree a power of $p$. Note that considering direct summands of motives of projective quadrics in Voevodsky's category of motives $\mathrm{DM}(F;\mz/2\mz)$, Tate traces corresponds to Vishik's isotropic motivic realization \cite{vishik:isotropic} and are given by Bachmann's generalized geometric fixed point functor \cite{bachmann,bachmannvishik}.

We actually show more: from Theorem \ref{main.thm}, higher Tate traces characterize isomorphism classes of objects in the additive category $\IM{F}{\mz/p\mz}$ generated by shifts of upper motives of geometrically split irreducible varieties satisfying the nilpotence principle (Definition \ref{definit}). The result thus applies to arbitrary direct summands of projective homogeneous varieties for $p$-inner semi-simple groups, our main focus in this article, but actually holds for a broader class of motives, containing Artin-Tate motives with respect to separable $p$-primary field extensions, motives of some projective pseudo-homogeneous varieties, and much more. Our results therefore highlight the question of determining the varieties suited for the theory of upper motives, that is, the varieties with motive isomorphic to a direct sum of initial motives.\\

We derive several applications from Theorem \ref{main.thm}. 
\begin{enumerate}
\item Computing the Tate trace of tensor products of motives, we prove a cancellation property in the category $\IM{F}{\mz/p\mz}$ (Corollary \ref{produits}).
\item As Tate traces of projective homogeneous varieties for $p$-inner groups are determined by the Tits indexes of the underlying semi-simple group, we show that their motives are determined by cohomological invariants in numerous situations (Corollaries \ref{critborel} and \ref{cohinv}).
\item Considering $G$, $G'$ two inner semi-simple groups of the same type, we relate motivic isomorphisms of projective homogeneous varieties for $G$ and $G'$ and their Tits indexes over field extensions (Theorem \ref{motequivpart}). This generalizes the main result of \cite{dC}, expressing motivic equivalence of semi-simple algebraic groups through their higher Tits $p$-indexes. The proof of Theorem \ref{motequivpart} relies on the theory of upper motives, Chernousov-Gille-Merkurjev motivic decompositions \cite{CGM} and their interpretation through Bialynicki-Birula-Hesselink-Iversen filtrations due to Brosnan \cite{Bros}. 
\item We introduce a notion of motivic splitting pattern and motivic splitting towers (Definition~\ref{splittingpattern.def}), which comprises the classical splitting pattern and Knebusch's tower of quadratic forms~\cite{Kne1}. 
\end{enumerate}

\textbf{Acknowledgments.} We are grateful to Skip Garibaldi, Stefan Gille, Nikita Karpenko, Srimathy Srinivasan and Alexander Vishik for their comments on a preliminary version of this paper, and insightful discussions.

\subsection{Notation} 
Let $F$ be a field and $\Lambda$ a commutative ring. Varieties over $F$, that is separated schemes of finite type over $\Spec F$, are always assumed to be smooth and projective. We denote by $\M_F(\Lambda)$ the category of Chow motives with coefficients in $\Lambda$, and use the notation $M(X)$ for the motive of a smooth complete scheme over $F$. 
Except in a few places where we consider motives with coefficients in $\mz$, we work with $\mz/p\mz$ coefficients for some prime number $p$, and we let $\Ch_{\ast} (\cdot):=\CH_{\ast}(\cdot)\otimes_{\mathbb{Z}}\mz/p\mz$ be the corresponding Chow groups. Our main reference for the basics of Chow groups and Chow motives is \cite{EKM}.

\section{Preliminary results} 
\label{prel.sec}

A {\em pure Tate motive} in $\M_F(\Lambda)$ is a direct sum of Tate motives, that is $M=\Sigma_{i\in I} \Lambda\{k_i\}$ for a finite set of integers $\{k_i,\ i\in I\}\subset \mz^I$, where $\Lambda\{k\}$ stands for the motive $M(\Spec F)\{k\}$. A motive $M\in\M_F(\Lambda)$ is called {\em geometrically split} if there exists a field extension $E/F$ such that $M_E$ is a pure Tate motive.
A variety $X$ over $F$ is called {\em geomerically split} with coefficients in $\Lambda$ if its motive $M(X)\in\M_F(\Lambda)$ is geometrically split. An irreducible variety $X$ over $F$ is called {\em generically split} if $M(X)_{F(X)}$ is a pure Tate motive. For instance, projective homogeneous varieties under the action of a semi-simple affine algebraic group $G$ are geometrically split by~\cite{Kock} for any ring of coefficients $\Lambda$; see also~\cite[66.5,\,66.7]{EKM} for explicit examples.

Let $M\in\M_F(\Lambda)$ be a motive. A decomposition of $M$ as a direct sum of indecomposable objects is called a \textit{complete} motivic decomposition of $M$. We say that the {\em Krull-Schmidt property} holds for $M$ if any decomposition of $M$ can be refined into a complete one, and $M$ admits a unique complete decomposition, up to permutation and isomorphism of the summands. 
By~\cite{vish-lens}, the Krull-Schmidt property holds for $M(X)\in\M_F(\mz)$ if $X$ is a projective quadric. Nevertheless, it does not hold in general for motives with coefficients in $\mz$, even for projective homogeneous varieties ; see~\cite[Ex. 32]{CM} for an explicit example using Severi-Brauer varieties. Therefore, we consider Chow motives with coefficients in $\mz/p\mz$ for some prime number $p$. In this setting, the Krull-Schmidt property holds for a large class of varieties, as we proceed to recall. 

We say that a variety $X$ satisfies the {\em nilpotence principle} if for any ring of coefficients $\Lambda$, an endomorphism $f$ of the motive $M(X)$ which maps to zero in $\End(M(X_L))$ for some field extension $L/F$ is nilpotent, see~\cite{Rost}. For example, projective homogeneous varieties under the action of an affine semi-simple algebraic group satisfy the nilpotence principle~\cite[Thm 8.2]{CGM}. 

By~\cite[\S 2.I]{K-inner}, if $X$ is geometrically split and satisfies the nilpotence principle, then the Krull-Schmidt property holds for $M(X)\in\M_F(\mz/p\mz)$. More precisely, we have the following: 
\begin{prop} \cite[Cor. 2.6]{K-inner}
\label{KS}
Let $\RM{F}{\mz/p\mz}$ be the full additive subcategory of $\M_F(\mz/p\mz)$ generated by direct  summands of shifts of motives of geometrically split varieties satisfying the nilpotence principle. The Krull-Schmidt property holds for all objects in $\RM{F}{\mz/p\mz}$.
 \end{prop}

\begin{proof} We reproduce the argument for the reader's convenience. By definition, all objects of $\RM{F}{\mz/p\mz}$ are finite sums of indecomposable ones, and an indecomposable object is isomorphic to 
$M\{i\}$, for some $i\in\mz$, and some indecomposable summand $M$ of $M(X)$, where $X$ is a geometrically split variety satisfying the nilpotence principle. Therefore, by~\cite[Thm 3.6 of Chapter I]{Bass}, we need to prove that the endomorphism rings of indecomposable objects are local rings. So, consider an indecomposable motive $M\{i\}$ as above and let $\pi\in\End(M(X))$ be the corresponding projector. We have \[\End(M\{i\})=\End(M)=\pi\End\bigl(M(X)\bigr)\pi\subset \End\bigl(M(X)\bigr).\] 
In particular, by~\cite[Cor. 2.2]{K-inner}, since we work with finite coefficients, for any $f\in\End(M\{i\})$, there exists a non-zero integer $n$ such that $f^n$ is a projector. As $M\{i\}$ is indecomposable, it follows that $f^n\in\{0,1\}$. Therefore, all non-invertible elements of $\End(M\{i\})$ are nilpotent, hence it is a local ring. 
\end{proof} 

Assume now that $X$ and $Y$ are two irreducible varieties of respective dimensions $d$ and $d'$. By definition, a morphism $M(X)\rightarrow M(Y)$ in the category $\M_F(\mz/p\mz)$ is a correspondence $\alpha:X \rightsquigarrow Y$, that is a cycle $\alpha\in \Ch_{d}(X\times Y)$. The projection $p_1$ on the first component induces a morphim ${p_1}_*:\,\Ch_{d}(X\times Y)\mapsto \Ch_{d}(X)\simeq \mz/p\mz \cdot [X]$. The multiplicity $\mult(\alpha)\in\mz/p\mz$ of $\alpha$ is defined by ${p_1}_*(\alpha)=\mult(\alpha)\cdot[X]$. 

A composition of correspondences $\alpha: X\rightsquigarrow Y$ and 
$\beta: Y\rightsquigarrow Z$ has multiplicity $\mult(\beta\circ \alpha)=\mult(\beta)\mult(\alpha)$, see~\cite[Cor. 1.7]{K-anisotropy}. 
In particular, the multiplicity of a projector $\pi:X\rightsquigarrow X$ is an idempotent of $\mz/p\mz$, hence $0$ or $1$. As explained in~\cite[\S 2.II]{K-inner}, if the diagonal $\Delta_X$ decomposes as a sum of pairwise orthogonal projectors, then one of them, called upper, has multiplicity $1$ and all others have multiplicity $0$. It follows that in a given complete decomposition of $M(X)$, there is a unique indecomposable summand defined by a projector of mutiplicity $1$. If in addition $X$ is geometrically split and satisfies the nilpotence principle, then $M(X)$ satisfies the Krull-Schmidt property, so that up to isomorphism, this summand does not depend on the choice of a decomposition. The (isomorphism class) of this motive is then called the \textit{upper motive} of $X$, and denoted by $U_X$.

The variety $X$ is called {\em isotropic mod $p$} if $X$ admits a $0$-cycle with $\mz$ coefficients of prime to $p$ degree, or, equivalently, a $0$-cycle with $\mz/p\mz$ coefficients of degree $1$. 
Such a cycle $\alpha\in\Ch_0(X)$ may also be viewed as a correspondence $\Spec F\rightsquigarrow X$ of multiplicity $1$. We get the following characterisation of isotropy: 
\begin{lem}
\label{isotropy}
Let $X$ be an irreducible geometrically split variety. 
The following are equivalent : 
\begin{enumerate} 
\item[(i)] $X$ is isotropic mod $p$; 
\item[(ii)] $M(X)$ contains a summand isomorphic to $\mz/p\mz\{0\}$;
\item[(iii)] $M(X)$ contains a summand isomorphic to $\mz/p\mz\{k\}$ for some $k\in\mz$. 
\end{enumerate} 
\end{lem} 
\begin{proof} 
Assume $X$ is isotropic mod $p$ and let $\alpha\in\Ch_0(X)$ be a cycle with $\deg(\alpha)=1$. 
Consider the product $[X]\times\alpha\in \Ch_d(X\times X)$. Viewed as a correspondence $X\rightsquigarrow X$, one may check it is idempotent, so it defines a summand $(X,[X]\times \alpha)$ of $M(X)$. Moreover, its multiplicity is $\mult([X]\times\alpha)=\deg(\alpha)=1$. The correspondences $\alpha:\Spec F\rightsquigarrow X$ and $[X]:X\rightsquigarrow \Spec F$ induce isomorphisms between $(X,[X]\times\alpha)$ and $\mz/p\mz\{0\}$, and we get (ii). 
The implication $(ii)\,\Rightarrow\,(iii)$ is obvious, and it remains to prove $(iii)\,\Rightarrow\,(i)$. 
Assume $M(X)$ contains a summand $(X,\pi)$ isomorphic to $\mz/p\mz\{k\}$, and let $\alpha:\, \mz/p\mz\{k\}\rightsquigarrow X$ and $\beta:X\rightsquigarrow \mz/p\mz\{k\}$ be correspondences inducing an isomorphism between $(X,\pi)$ and $\mz/p\mz\{k\}$. 
We have \[\alpha\in\Ch_k (X)\simeq \Ch^{d-k}(\Spec F\times X)\mbox{ and }\beta\in\Ch_{d-k}(X)\simeq \Ch^{k}(X\times \Spec F).\] Moreover,  the intersection product of their pull-backs to $\Spec F\times X\times \Spec F$, along the projections $p_{12}$ and $p_{23}$, respectively, is a cycle $\gamma\in\Ch^d(X)\simeq \Ch_0(X)$. 
Since $\beta\circ\alpha=1\in\Ch_0(\Spec F)\simeq\mz/p\mz$, the $0$-cycle $\gamma$ has degree $1$ and this proves that $X$ is isotropic mod $p$. 
\end{proof} 

Let $X$ and $Y$ be two irreducible varieties. There is a correspondence $\alpha:X\rightsquigarrow Y$ of multiplicity $1$ if and only if $Y_{F(X)}$ is isotropic mod $p$, see for instance~\cite[Cor. 57.11, Lem. 75.1]{EKM}. When these conditions are satisfied, we say that  $X$ {\em dominates $Y$ mod $p$}, and we write $X\succeq_p Y$, see~\cite[Def. 4]{dC}. 
The domination relation is a preorder and the associated equivalence relation is called {\em equivalence mod $p$} and denoted by $\approx_p$. 
\begin{ex}
Assume $X$ and $Y$ are quadrics. By Springer's theorem~\cite[Cor. 71.3]{EKM}, $Y_{F(X)}$ is isotropic mod $2$ if and only if there is a rational map $X\dashrightarrow Y$. Hence, for quadrics, the domination relation mod $2$ coincides with the relation considered by Bruno Kahn in~\cite[\S I.7]{Kahn-Bourbaki}, and the corresponding equivalence relation is the stable birational equivalence. Therefore, we may view the equivalence relation mod $p$ as a motivic version of stable birational equivalence. 

\end{ex} 
By definition, $X\approx_p Y$ if and only if there are correspondences $X\rightsquigarrow Y$ and $Y\rightsquigarrow X$ of multiplicity $1$, or equivalently, $X_{F(Y)}$ and $Y_{F(X)}$ are isotropic mod $p$. 
For the varieties we are interested in, the following result, which is due to Karpenko, provides another characterization in terms of upper motives: 
\begin{lem}\cite[Cor. 2.15]{K-inner}
\label{upper.lem}
Let $X$ and $Y$ be two irreducible geometrically split $F$-varieties satisfying the nilpotence principle. The varieties $X$ and $Y$ are equivalent mod $p$ if and only if the upper motives $U_X$ and $U_Y$ are isomorphic in $\M_F(\mz/p\mz)$. 
 \end{lem} 

\section{Tate trace and initial motives} 

Let $M$ be a motive in $\RM{F}{\mz/p\mz}$. 
By definition, $M$ is geometrically split; a field $E/F$ such that $M_E$ is a pure Tate motive is called a splitting field of $M$. 
The decomposition as a sum of Tate motives of $M_E\in\RM{E}{\mz/p\mz}$ does not depend on the choice of a splitting field $E$. Hence, we may give the following definition:

\begin{defi} \begin{enumerate}
\item The rank $\rk(M)$ of a motive $M\in\RM{F}{\mz/p\mz}$ is the rank of $\Ch_*(M_E)$ as a $\mz/p\mz$-module, where $E/F$ is a splitting field for $M$. 
\item The hook $h(M)$ of a non-trivial motive $M\in\RM{F}{\mz/p\mz}$ is the smallest integer $i\in\mz$ such that $\Ch^i(M_E)$ is non-zero, where $E/F$ is a splitting field for $M$. 
\end{enumerate}
\end{defi}
\begin{ex}
\label{Taterankhook} 
Assume $T$ is a non-trivial pure Tate motive. The rank of $T$ is the number of Tate summands in a complete decomposition of $T$, while its hook $h(T)$ is the smallest integer $i$ such that $\mz/p\mz\{i\}$ is a direct summand of $T$. This follows from the fact that $\Hom_{\M_F(\mz/p\mz)}(M,\mz/p\mz\{i\})=\Ch^i(M)$ and 
\[\Hom_{\M_F(\mz/p\mz)}(\mz/p\mz\{k\},\mz/p\mz\{i\})\mbox{ is }0\mbox{ if }k\not =i\mbox{ and }\mz/p\mz\mbox{ otherwise}.\]
%$\Ch_i(\mz/p\mz\{k\})$ is $0$ is $i\not=k$ and $\mz/p\mz$ otherwise. 
\end{ex}
\begin{ex}
Let $X$ be an irreducible geometrically split variety satisfying the nilpotence principle. The hook of the motive of $X$ is $h\bigl(M(X)\bigr)=0$.  
Moreover, we also have $h(U_X)=0$, where $U_X$ is the upper motive of $X$, see~\cite[Lem. 2.8]{K-inner}. 
\end{ex}
\begin{ex}
\label{upperhook}
Given two motives $M, N \in\RM{F}{\mz/p\mz}$, one may easily check that
\begin{multline*}
\rk(M+N)=\rk(M)+\rk(N),\ h(M+N)=\min\{h(M),h(N)\},\\\rk(M\otimes N)=\rk(M)\rk(N)\mbox{  and }h(M\otimes N)=h(M)+h(N).
\end{multline*}
In particular, for all $i\in\mz$, we have $h(M\{i\})=h(M)+i$, and for all irreducible geometrically split variety $X$ satisfying the nilpotence principle, $h(U_X\{i\})=i$. 
\end{ex}
Let $M\in\RM{F}{\mz/p\mz}$ be a motive.  By the Krull-Schmidt property, see Prop~\ref{KS}, two pure Tate summands of $M$ of maximal rank are isomorphic. Therefore, we may give the following definition: 
\begin{defi}
A pure Tate motive is called the Tate trace of $M\in\RM{F}{\mz/p\mz}$ if it is isomorphic to a pure Tate summand of maximal rank of $M$. The Tate trace of $M$ is uniquely defined up to isomorphism and denoted by $\Tr(M)$. 
\end{defi}
\begin{ex}
\label{TateVariety.ex}
A motive $M$ is a pure Tate motive if and only if $\Tr(M)\simeq M$. \\
By Lemma~\ref{isotropy}, if $X$ is an irreducible geometrically split variety satisfying the nilpotence principle, then $\Tr\bigl(M(X)\bigr)\not\simeq 0$ if and only if $X$ is isotropic mod $p$. In particular, the Tate trace does not commute with scalar extension. 
\end{ex}
\begin{ex}
\label{qf.ex}
Let $q=i_0{\mathbb {H}}+q_{\mathrm {an}}$ be a $2m$-dimensional quadratic form with Witt index $i_0$ and anisotropic part $q_{an}$. Denote by $Q$ and $Q_{\mathrm{an}}$ the corresponding projective quadrics. By~\cite{Rost}, the motive $M(Q)\in\M_F(\mz/2\mz)$ satisfies: 
\begin{multline*}
    M(Q)\simeq \mz/2\mz\{0\}+\dots+\mz/2\mz\{i_0-1\}+M(Q_{\mathrm{an}})\{i_0\}\\
+\mz/2\mz\{2m-1-i_0\}+\dots+\mz/2\mz\{2m-2\}.
\end{multline*}
By Springer's theorem~\cite[71.3]{EKM}, $Q_{\mathrm{an}}$ is anisotropic mod $2$ so that $M(Q_{\mathrm{an}})$ has trivial Tate trace. Therefore, the Tate trace of $M(Q)$ is given by \[\Tr(M(Q))\simeq\sum_{k=0}^{i_0-1} \bigl(\mz/2\mz\{k\}+\mz/2\mz\{2m-2-k\}\bigr).\]
\end{ex}

\begin{defi}
  A motive $M$ in $\RM{F}{\mz/p\mz}$ is called isotropic if it has a non-trivial Tate trace, and anisotropic otherwise.
\end{defi}

By the Krull-Schmidt property, $M$ decomposes in a unique way as $M=\Tr(M)+M_{an}$ for some anisotropic motive $M_{an}$ uniquely defined up to isomorphism, called the anisotropic part of $M$.
The Tate trace of a motive measures its splitting. 
\begin{ex}
\label{oi.ex}
Assume $F$ has characteristic different from $2$ and let $A$ be an $F$-central simple algebra endowed with an orthogonal involution $\sigma$. The Witt index $i_w(\sigma)$ of the involution $\sigma$ is the reduced dimension of a maximal isotropic right ideal in $A$~\cite[\S 6.A]{KMRT}. 
The variety of isotropic right ideals of reduced dimension $1$ is called the involution variety and denoted by $X_\sigma$. It is a $\PGO^+(A,\sigma)$-projective homogeneous variety, hence $M(X_\sigma)\in\RM{F}{\mz/2\mz}$. 
We claim that its Tate trace is given by
\[\Tr(M(X_\sigma))\simeq\left\{\begin{array}{ll}
0&\mbox{ if $A$ is non split}, \\
\sum_{k=0}^{i_w(\sigma)-1} \bigl(\mz/2\mz\{k\}+\mz/2\mz\{2m-2-k\}\bigr)&\mbox{ if $A$ is split.} 
\end{array}\right.\]
Indeed, the algebra $A$ contains a right ideal of reduced dimension $1$ if and only if it is split. Moreover, since $A$ admits an orthogonal involution, it has exponent $2$ and its Schur index is invariant under any odd degree extension of the base field. Therefore, if $A$ is non-split, then the variety $X_\sigma$ is anisotropic mod $2$ and this proves the first equality. 
Assume now that $A$ is split. Then $\sigma$ is adjoint to a quadratic form $q$ over $F$ and it follows from the definitions that $X_\sigma$ is isomorphic to the corresponding quadric $Q$ and $i_w(\sigma)$ is equal to the Witt index of $q$. Hence the second equality follows from Example~\ref{qf.ex}.  
\end{ex}
We may characterize the domination relation mod $p$ in terms of Tate traces as follows: 
\begin{prop}
Let $X$ and $Y$ be two irreducible geometrically split varieties satisfying the nilpotence principle. Then $X\succeq_p Y$ if and only if $\Tr\bigl(M(Y_L)\bigr)\not\simeq 0$ for all field extensions $L/F$ such that $\Tr\bigl(M(X_L)\bigr)\not\simeq0$. 
\end{prop}
\begin{proof}
Assume first that $X\succeq_p Y$. By definition there exists a correspondence $\alpha:X \rightsquigarrow Y$ of multiplicity $1$. Hence, any correspondence $Z \rightsquigarrow X$ provides, by composition with $\alpha$, a correspondence $Z\rightsquigarrow Y$ of the same multiplicity. 
So, $Y_L$ is isotropic mod $p$ for all $L/F$ such that $X_L$ is isotropic mod $p$. By example~\ref{TateVariety.ex}, this proves the first implication.  
To prove the converse, it suffices to consider the field $L=F(X)$. Since $X_{F(X)}$ is isotropic mod $p$, the corresponding motive $M(X_{F(X)})$ has non-zero Tate trace. Therefore $\Tr\bigl(M(Y_{F(X)})\bigr)\not=0$ and this proves $Y_{F(X)}$ is isotropic mod $p$ and  $X\succeq_p Y$, which finishes the proof. 
\end{proof} 
We now extend the domination relation to objects of $\RM{F}{\mz/p\mz}$. 
\begin{defi} Let $M,N\in\RM{F}{\mz/p\mz}$ be two motives. 
\begin{enumerate} 
\item We say that $M$ dominates $N$, and we write $M\succeq N$, if $\Tr(N_L)\not\simeq 0$ for all field extension $L/F$, such that $\Tr(M_L)\not\simeq 0$. 
\item We say that $M$ is equivalent to $N$, and we write $M\approx N$ if $M\succeq N$ and $N\succeq M$. 
\end{enumerate} 
\end{defi}
Two motives are equivalent if and only if they are isotropic over the same field extensions of the base field. For instance, two isotropic motives are equivalent. 

We now consider a subcategory of $\RM{F}{\mz/p\mz}$, which, as recalled below, contains the motives we are interested in, and for which we will be able to prove an isomorphism criterion for objects using the Tate trace (see Theorem~\ref{main.thm}). 
\begin{defi} \label{definit}
\begin{enumerate}
\item A motive $P\in\RM{F}{\mz/p\mz}$ is called {\em initial} if there exists an irreducible geometrically split variety $X$ satisfying the nilpotence principle such that $P$ is isomorphic to $U_X\{i\}$ for some $i\in\mz$, where $U_X$ is the upper motive of $X$. 
\item The full additive subcategory of $\RM{F}{\mz/p\mz}$ generated by initial motives is denoted by $\IM{F}{\mz/p\mz}$. 
\end{enumerate}
\end{defi}

\begin{ex}\label{exproj}
    Let $G$ be a semi-simple affine algebraic group over $F$ and let $E/F$ be a minimal field extension such that $G_E$ is of inner type. Fix an intermediate field extension $E/K/F$ and a projective $G_K$-homogeneous variety $X$. The corestriction of $X$ to $F$ corresponds to the scheme $X$, viewed as an $F$-variety through the morphism $\Spec(K)\longrightarrow \Spec(F)$. 
    In this paper, we refer to such varieties as quasi-$G$-homogeneous varieties\footnote{Quasi-homogeneous varieties were initially considered by Karpenko in~\cite{K-outer}. Note that in Karpenko's paper, a quasi-homogeneous variety is defined as a variety whose connected components are corestrictions of $G_L$-homogeneous varieties for various $G$ and $L$.}.  
   By~\cite[\S 2]{K-outer}, $X$ is geometrically split and satisfies the nilpotence principle.
   Assume now $G$ is a semi-simple affine algebraic group over $F$ of $p$-inner type, i.e. $G_E$ is inner for some field extension $E/F$ of $p$-power degree. By \cite[§1]{K-outer}, the motive of a projective $G$-homogeneous variety decomposes in $\M_F(\mz/p\mz)$ as a direct sum of $U_{X_i}\{k_i\}$ for some quasi-$G$-homogeneous varieties $X_i$ and some $k_i\in \mz$. In particular, if $G$ is $p$-inner, motives of projective $G$-homogeneous varieties with coefficients in $\mz/p\mz$ belong to $\IM{F}{\mz/p\mz}$.

\end{ex}
\begin{rem}
\label{outer.rem}
For arbitrary semi-simple algebraic groups, motives of projective homogeneous varieties do not belong to $\IM{F}{\mz/p\mz}$ (see \cite[Example 3.3]{K-outer}). The techniques of the present article are thus only suited for $p$-inner groups and in a subsequent work (joint with N. Karpenko) we construct counterexamples to Theorem \ref{main.thm} out of this context.
\end{rem} 

\begin{rem}
Let $X$ be a geometrically irreducible, geometrically split variety satisfying the nilpotence principle. After extension to a field $L/F$, the upper motive $U_X$ of $X$ is still a direct summand of $M(X_L)$ defined by a projector of multiplicity $1$, but it may not be indecomposable anymore. Therefore, $U_{(X_L)}$ is a direct summand of ${(U_X)}_L$ but they need not be isomorphic. Given a motive  $M$ in $\IM{F}{\mz/p\mz}$, it is not known in general whether or not $M_L\in \IM{L}{\mz/p\mz}$. 
\end{rem}
A nice property of $\IM{F}{\mz/p\mz}$ is that the hook of an element is easy to compute from a complete decomposition: by example~\ref{upperhook}, if 
$M=\sum_{k=1}^rU_{X_k}\{i_k\}$, then $h(M)={\mathrm {min}}\{i_k,\ k\in\llbracket 1,r\rrbracket\}$. In addition, we have the following isomorphism criterion for indecomposable motives: 
\begin{lem} 
\label{indec.iso}
Let $M=U_X\{i\}$ and $N=U_Y\{j\}$ be two initial motives. Then $M\simeq N$ if and only if $i=j$ and $X\approx_p Y$.  
\end{lem} 
\begin{proof}
If $M$ and $N$ are isomorphic, they have the same hook, so we have $i=j$. Hence, $M$ and $N$ are isomorphic if and only if $i=j$ and $U_X$ and $U_Y$ are isomorphic. Lemma~\ref{upper.lem} finishes the proof. 
\end{proof} 
We may also compute the Tate trace of a tensor product of two motives $M$ and $N$ in $\IM{F}{\mz/p\mz}$: 
\begin{lem}
\label{Tateproduct}
For all $N,M\in\IM{F}{\mz/p\mz}$, we have $\Tr(M\otimes N)=\Tr(M)\otimes \Tr(N)$.  
\end{lem} 
\begin{proof}
Starting from decompositions $M=\sum_{k=1}^rU_{X_k}\{i_k\}$ and $N=\sum_{\ell=1}^sU_{Y_\ell}\{j_\ell\}$ as sums of initial motives, we get 
\[M\otimes N\simeq\sum_{1\leq k\leq r,\ 1\leq \ell\leq s} \bigl(U_{X_k}\otimes U_{Y_\ell}\bigr)\{i_k+j_\ell\}.\]

Given two irreducible geometrically split varieties satisfying the nilpotence principle $X$ and $Y$, we claim that the Tate trace of $U_{X}\otimes U_{Y}$ is non zero if and only if both factors are Tate motives. Indeed, $U_X\otimes U_Y$ is a direct summand of $M(X\times Y)$. Therefore, if $\Tr(U_X\otimes U_Y)\not \simeq0$, then $X\times Y$ is isotropic mod $p$. Hence both $X$ and $Y$ are isotropic mod $p$, and by Lemma~\ref{isotropy}, we get that $U_X\simeq \mz/p\mz\{0\}\simeq U_Y$. 
As a consequence, the Tate factors in a complete decomposition of $M\otimes N$ all occur as a tensor product of a Tate factor in a complete decomposition of $M$ and a Tate factor in a complete decomposition of $N$. 
 \end{proof} 

\section{Motivic isomorphisms and partial splitting fields}

In this section, we state and prove the main result of this article, which shows how Tate traces over field extensions determine isomorphism classes in $\IM{F}{\mz/p\mz}$. Here is the supporting terminology.

\begin{defi}Let $M$ be a motive in $\IM{F}{\mz/p\mz}$ and write $M\simeq \sum_{k=1}^rU_{X_k}\{i_k\}$ as a direct sum of initial motives. The function fields $F(X_1),...,F(X_r)$ are called partial splitting fields of the motive $M$, and we say that $\mathcal{C}=\{F(X_1),...,F(X_r)\}$ is a partial splitting family of $M$.
\end{defi}

By the Krull-Schmidt property, two partial splitting families of a motive $M$ have the same cardinality. Moreover, by Lemma \ref{indec.iso}, given two such families $\mathcal{C}=\{F(X_1),...,F(X_r)\}$ and $\mathcal{C}'=\{F(Y_1),...,F(Y_r)\}$, and renumbering one of them if necessary, we may assume that for all $1\leq i\leq r$, $X_i$ and $Y_i$ are equivalent mod $p$.

\begin{ex}
\label{phv.ex}
Let $G$ be a $p$-inner semi-simple algebraic group. Fix a minimal field extension $E/F$ such that $G_E$ is of inner type. By Example \ref{exproj}, $M(X)$ has a splitting family given by several function fields of quasi-$G$-homogeneous varieties.
\end{ex}

We now prove that Tate traces over partial splitting fields detect isomorphisms in $\IM{F}{\mz/p\mz}$. 

\begin{thm}
\label{main.thm}
Let $M$ and $N$ be two motives in $\IM{F}{\mz/p\mz}$. The following assertions are equivalent:
\begin{enumerate} 
\item[(i)] $M$ and $N$ are isomorphic; 
\item[(ii)] $M_L$ and $N_L$ have isomorphic Tate traces for all field extensions $L/F$;
\item[(iii)] There is a splitting family $\mathcal{C}_M$ of $M$ and a splitting family $\mathcal{C}_N$ of $N$ such that $M$ and $N$ have isomorphic Tate traces over all fields in $\mathcal{C}_M\cup\mathcal{C}_N$. 
\end{enumerate} 
\end{thm} 

\begin{rem} 
When condition $(ii)$ holds, we say that $M$ and $N$ have isomorphic higher Tate trace. 
\end{rem} 
\begin{proof} 
Since $\IM{F}{\mz/p\mz}$ is a full subcategory of $\RM{F}{\mz/p\mz}$, where the Tate trace is defined, we only have to show $(iii)\Rightarrow (i)$.

Assume that $M$ and $N$ are objects of $\IM{F}{\mz/p\mz}$, and fix $\mathcal{C}_M$, $\mathcal{C}_N$ two respective splitting families given by motivic decompositions 
\begin{equation}
  M\simeq\sum_{k=1}^nU_{X_k}\{i_k\}\mbox{ and }N\simeq\sum_{\ell=1}^mU_{Y_\ell}\{j_\ell\}.  
\end{equation}

If $M$ has trivial rank, by the nilpotence principle, it is trivial. 
We claim that $N$ also is trivial in this case. Indeed, if $N$ is non-trivial, the splitting family ${\mathcal C}_N$ is non empty, and $N$, hence also $M$, have non-zero Tate trace over all fields of ${\mathcal C}_N$. 

So we proceed by induction on the minimum $r$ of the ranks of $M$ and $N$ and we may assume $r\geq 1$. 
Permuting the summands in $(1)$ if necessary, we may order them by increasing hooks, so that $h(M)=i_1$ and $h(N)=j_1$. Since $M$ contains a summand isomorphic to $\mz/p\mz\{i_1\}$ over the field $F(X_1)$ which belongs to ${\mathcal C}_M$, condition $(iii)$ guarantees that $h(N)=j_1\leq i_1$. Using a similar process with $F(Y_1)$, we get that $M$ and $N$ have the same hook $h$ and tensoring $M$ and $N$ by $\mz/p\mz\{-h\}$, we may assume they both have hook $0$.

Let $M_0$ (respectively $N_0$) be the sum of all indecomposable summands in the complete decomposition of $M$ (respectively $N$) in $(1)$ having hook $0$. The motives $M_0$ and $N_0$ inherit from $(1)$ decompositions $M_0=\sum_{k=1}^s U_{X_k}$ and  $N_0=\sum_{\ell=1}^tU_{Y_\ell}$ for some non-zero integers $s$ and $t$.

Consider the function field $F_1=F(X_1)$. The variety $X_1$ becomes isotropic mod $p$ over $F_1$, therefore by Lemma~\ref{isotropy}, $(M_0)_{F_1}$ contains a summand isomorphic to $\mz/p\mz\{0\}$. Since $F_1$ belongs to $\mathcal{C}_M$, condition $(iii)$ implies that $\mz/p\mz\{0\}$ also is a summand of $N_{F_1}$. By definition of $N_0$, we have $N=N_0+N_1$ with $h(N_1)\geq 1$, so that $\mz/p\mz\{0\}$ actually is a summand of $(N_0)_{F_1}$. It follows that $\mz/p\mz\{0\}$ is a summand of $(U_{Y_{\ell_1}})_{F_1}$ for some $\ell_1\in\llbracket 1,t\rrbracket$, hence of $M(Y_{\ell_1})_{F_1}$. Applying again Lemma~\ref{isotropy}, we get that $(Y_{\ell_1})_{F_1}$ is isotropic mod $p$, that is $X_1$ dominates $Y_{\ell_1}$ mod $p$. The same reasoning with the function field $F(Y_{\ell_1})$ shows that the variety $Y_{\ell_1}$ dominates $X_{k_1}$, for some $k_1\in\llbracket 1,s\rrbracket$.
The integer $k_1$ need not be equal to $1$; but iterating the process, we get a decreasing chain of varieties 
\[X_1\succeq_{p} Y_{\ell_1}\succeq_p X_{k_1}\succeq_p Y_{\ell_2}\succeq_p\dots\succeq_p X_{k_n},\] for some integers $k_i\in\llbracket 1,s\rrbracket$ and $\ell_i\in\llbracket 1,t \rrbracket $. 
Since the set of varieties $X_{k}$ and $Y_{\ell}$ are finite, there exists $k\in\llbracket 1,s\rrbracket$ and $\ell\in\llbracket 1,t\rrbracket$ such that $X_k\succeq_p Y_\ell \succeq_p X_k$, hence, $X_k\approx_p Y_\ell$. By Lemma~\ref{indec.iso}, it follows that 
$U_{X_k}$ and $U_{Y_\ell}$ are isomorphic, so we may write $M\simeq U_{X_k}+M'$ and $N\simeq U_{X_k}+N'$ for some motives $M'$ and $N'$ in  $\IM{F}{\mz/p\mz}$. As the motives $M'$ and $N'$ have partial splitting families contained in ${\mathcal C}_M $ and ${\mathcal C}_N$ and both have rank strictly lesser than $r$, we conclude by induction.
\end{proof} 

\begin{rem}
As pointed out by Stefan Gille, the proof of Theorem \ref{main.thm} carries on replacing the Tate traces of $N_L$ and $M_L$ by their largest $0$-dimensional direct summands. We would call these coarser invariants their Artin-Tate Trace.
\end{rem}

As explained above, our theorem applies with $\mz/p\mz$ coefficients to all direct summands of motives of projective homogeneous varieties under the action of $p$-inner semi-simple algebraic groups. In this setting, Theorem~\ref{main.thm} recovers numerous known results, see the following and \S\,\ref{application}. 
For quadrics, it provides a new proof of Vishik's criterion for motivic equivalence: 
\begin{cor}[Vishik~\cite{vish-lens}]
\label{Vishik}
Let $Q$ and $Q'$ be smooth projective quadrics corresponding to quadratic forms $q$ and $q'$ of the same dimension. 
The motives of $Q$ and $Q'$ are isomorphic in $\M_F(\mz/2\mz)$  if and only if the quadratic forms $q_L$ and $q'_L$ have the same Witt index for all field extension $L/F$.  
\end{cor}
\begin{proof}
Rost's decomposition~\cite[Prop. 2]{Rost} states that for any field extension $L/F$, the Tate trace of the motive $M(Q_L)$ (resp. $M(Q'_L)$) determines and is determined by the Witt index of $q_L$ (resp. of $q'_L$) and the dimension of $Q$ and $Q'$. The corollary then follows immediately from Theorem~\ref{main.thm}. 
\end{proof} 
The same argument applied to involution varieties, whose Tate traces are described in Example~\ref{oi.ex}, leads to the following generalization: 
\begin{cor}
\label{IV.cor}
Assume $F$ has characteristic different from $2$. 
Let $A$ be a central simple $F$-algebra and $\sigma$ and $\tau$ two orthogonal involutions on $A$. 
The motives of the corresponding involution varieties $X_\sigma$ and $X_\tau$ are isomorphic in $\M_F(\mz/2\mz)$ if and only if for all field extensions $L/F$ such that $A_L$ is a split algebra, we have $i_w(\sigma_L)=i_w(\tau_L)$. 
\end{cor}

\begin{rem}
 Corollary~\ref{IV.cor}, combined with Karpenko's theorem on anisotropy of orthogonal involutions after generic splitting \cite{K-anisotropy} and \cite{dC}, provides a new proof that involution varieties are critical, see \cite[Thm. 3.6]{dCQZ}.
 \end{rem}

\begin{ex}\label{SBex} Consider two central simple $F$-algebras $A$ and $B$ of the same degree $n$. The corresponding Severi-Brauer varieties $\SB(A)$ and $\SB(B)$ have isomorphic motives with coefficients in $\mz/p\mz$ if and only if their $p$-primary components $A_p$ and $B_p$ generate the same subgroup in the Brauer group of $F$. This can be deduced from the results of \cite{K-inner} or shown as follows. By \cite{K-motif}, the motives of $\SB(A)$ and $\SB(B)$ decompose, respectively, as direct sums of shifts of the motives of the Severi-Brauer varieties of $A_p$ and $B_p$. It follows that given a field extension $E/F$, the motive $M(\SB(A_E))$ (resp. $M(\SB(B_E))$) is either a pure Tate motive isomorphic to $M({\mathbb P}^{n-1}_E)$ or anisotropic, depending on whether the $p$-primary component $(A_p)_E$ (resp. $(B_p)_E$) is split or not. The conclusion is implied by Theorem~\ref{main.thm}, since two central simple algebras are split by the same field extensions of $F$ if and only if they generate the same subgroup of the Brauer group of $F$, by \cite[Theorem 9.3]{amitsur}.
\end{ex}
\begin{rem}
\label{Z.rem}
 As opposed to Vishik's criterion (Cor.~\ref{Vishik}), Theorem~\ref{main.thm} does not hold for motives with $\mz$ coefficients. Indeed, consider two central division $F$-algebras $D$ and $D'$ of degree $n$ and let $\SB(D)$ and $\SB(D')$ be their respective Severi-Brauer varieties. Assume $D$ and $D'$ generate the same subgroup of the Brauer-group, and yet $D$ is isomorphic neither to $D'$ nor to its opposite algebra. By~\cite[Criterion 7.1]{K:motivic}, the motives of $\SB(D))$ and $\SB(D')$ are not isomorphic in $\M_F(\mz)$; nevertheless, by theorems of Ch\^atelet and Amitsur (see~\cite[5.1.3, 5.4.1]{GS}), for any field extension $E/F$, either $M(\SB(D))$ and $M(\SB(D'))$ are anisotropic over $E$, or both are pure Tate motives isomorphic to $M({\mathbb{P}}_E^{n-1})$.
\end{rem}

We conclude this section with a cancellation property in $\IM{F}{\mz/p\mz}$: 
\begin{cor} \label{produits}
Let $M, N, N'\in\IM{F}{\mz/p\mz}$ be three non-trivial motives. 

If $N\succeq M$, $N'\succeq M$, and $M\otimes N\simeq M\otimes N'$, then $N\simeq N'$. 
\end{cor} 
\begin{proof} 
Consider $M$, $N$ and $N'$ as above and let $L/F$ be a field extension. Since $M\otimes N\simeq M\otimes N'$, they have the same higher Tate trace. By Lemma~\ref{Tateproduct}, we get that 
$\Tr(M_L)\otimes\Tr(N_L)\simeq \Tr(M_L)\otimes \Tr(N'_L)$. 

Assume first that $\Tr(N_L)=0$. Then $\Tr(M_L)\otimes \Tr(N'_L)=0$. Hence, either $\Tr(M_L)=0$ or $\Tr(N'_L)=0$. Since $N'\succeq M$, we actually get $\Tr(N'_L)=0$ in both cases. 

Assume now that $\Tr(N_L)\not=0$. Since $N\succeq M$, we also have $\Tr(M_L)\not =0$. In particular, $\Tr(M_L)$ has non-zero rank, and it follows that $\Tr(N_L)$ and $\Tr(N'_L)$ have the same rank $r$, see Example~\ref{upperhook}. Let $T$ be a pure Tate motive of maximal rank which is a direct summand of both $\Tr(N_L)$ and $\Tr(N'_L)$. We may decompose $\Tr(N_L)=T+S$ and $\Tr(N'_L)=T+S'$ for some pure Tate motives $S$ and $S'$ that contain no isomorphic Tate summand. By the Krull-Schmidt property, we have 
$\Tr(M_L)\otimes S\simeq \Tr(M_L)\otimes S'$. Assume for the sake of contradiction that $rk(T)<r$, so that $S$ and $S'$ are non-zero. Then the hook of $\Tr(M_L)\otimes S\simeq \Tr(M_L)\otimes S'$ is $h(\Tr(M_L))+h(S)=h(\Tr(M_L))+h(S')$. It follows that $S$ and $S'$ have the same hook $h$. Since they are pure Tate motives, they both contain a summand isomorphic to $\mz/p\mz\{h\}$ and this contradicts the maximality of $T$. 

So, we have proved that $N$ and $N'$ have the same higher Tate trace, and by Theorem~\ref{main.thm}, we get that $N\simeq N'$. 
\end{proof} 
\begin{rem}
Let $D$ and $D'$ be two division algebras of degree $n$. Assume the Brauer class of $D'$ belongs to the subgroup of the Brauer group generated by the class of $D$. Then the projection $\SB(D)\times \SB(D')\rightarrow \SB(D)$ is a projective bundle, and  for an arbitrary ring $\Lambda$, we have 
\[M(\SB(D)\times \SB(D'))\simeq \sum_{i=0}^{n-1} M(\SB(D))\{i\}\mbox{ in }\M_F(\Lambda).\]
It follows that $M(\SB(D)\times \SB(D'))\simeq M(\SB(D)\times \SB(D))$ in $\M_F(\mz)$. Applied to $D$ and $D'$ as in Remark~\ref{Z.rem}, this shows that the Corollary~\ref{produits} does not hold with coefficients in $\mz$. 
In addition,  in $\M_F(\mz/p\mz)$, we also have $M(\SB(D)\times \SB(D))\simeq M(\SB(D)\times M({\mathbb P}_F^{n-1})$. This shows the domination conditions are necessary in Corollary \ref{produits}. 
\end{rem} 

\section{Application to projective homogeneous varieties} 
\label{application}

\subsection{Motives of projective homogeneous varieties}
We now explore consequences of our main theorem in the setting of projective homogeneous varieties. 
Let $G$ be a semi-simple algebraic group over $F$, $T\subset G$ a maximal torus, and $\Delta(G)$ the Dynkin diagram of $G$. Its set of vertices, also denoted by $\Delta(G)$, is given by 
a set of simple roots of the root system of $G_{F_{sep}}$ with respect to $T_{F_{sep}}$. The absolute Galois group $\Gal(F_{sep}/F)$ acts on the Dynkin diagram $\Delta(G)$ via the $\ast$-action. We assume $G$ is of $p$-inner type, that is the $\ast$-action becomes trivial over a field extension of $p$-power degree. 

A $G$-projective homogeneous variety $X$ is a $G$-variety over $F$, isomorphic over a separable closure $F_{sep}$ to a quotient $G_{F_{sep}}/P$ by a parabolic subgroup. Given a subset $\Theta$ of  $\Delta(G)$, there is a parabolic subgroup $P_{\Theta}$ of $G_{F_{sep}}$ whose Levi component has Dynkin diagram $\Delta(G)\setminus \Theta$. Note that in the literature, there are two opposite conventions for parabolic subgroups of type $\Theta$; here, a Borel subgroup has type $\Delta(G)$. The variety $G_{F_{sep}}/P_{\Theta}$ is defined over $F$ if and only if $\Theta$ is $\Gal(F_{sep}/F)$-invariant, leading to a bijection $\Theta \leftrightarrow X_{\Theta,G}$ between $\ast$-invariant subsets of $\Delta(G)$ and isomorphism classes of $G$-projective homogeneous varieties defined over $F$.

\begin{cor}\label{c1}
Let $X$ and $X'$ be two projective homogeneous varieties for $p$-inner semi-simple groups $G$ and $G'$, respectively.  Fix a minimal field extension $E/F$ such that both $G_E$ and $G'_E$ are of inner type.
Two direct summands $M$ of $M(X)$ and $N$ of $M(X')$ in $\M_F(\mz/p\mz)$ are isomorphic if and only if their Tate traces are isomorphic over the function fields of the corestrictions to $F$ of projective homogeneous varieties for $G_L$ and $G'_L$, for all intermediate field extensions $E/L/F$.
\end{cor}
The result follows immediatly from Theorem~\ref{main.thm} since $M$ and $N$ both admit splitting families contained in the family of fields considered in the corollary, see Example~\ref{phv.ex}. 
Note that this family of fields is not minimal; but its definition depends only on the groups $G$ and $G'$, and not on a decomposition of $M$ and $N$ as direct sums of indecomposable summands.

\begin{rem}Corollary \ref{c1} extends to the case of projective pseudo-homogeneous varieties (that is, for non-reduced parabolic subgroups) over perfect fields of characteristic $p>3$ by \cite{Srini}.
\end{rem}

Assume that $M(X_{\Theta,G})$ is a pure Tate motive for some $\Gal(F_{sep}/F)$-invariant subset $\Theta\subset \Delta(G)$. The $\mz/p\mz$-module $\Ch(X_{\Theta,G})$ is free and the isomorphism class of $M(X_{\Theta,G})$ is determined by the {\em Poincaré polynomial}
\[P\bigl(M(X_{\Theta,G}),t\bigr)=\sum_{i\geq 0}a_it^i,\]
where $a_i:=\dim(\Ch_i(X_{\Theta,G}))$. 
The motive $M(X_{\Theta,G})$ being pure Tate, this polynomial is invariant under field extensions and equal to $P(M(X_{\Theta,G_{F_{sep}}}),t)$. Hence, it does not depend on the coefficient ring $\mz/p\mz$, and admits a description in terms of the Weyl group of $G$ and the length function, see~\cite{Kock}. We call it the Poincaré polynomial of $X_{\Theta,G_{F_{sep}}}$.
The next corollary applies, notably but not exclusively, to varieties of Borel subgroups of a semi-simple group of inner type. 
\begin{cor}\label{critborel}
Let $X$ and $X'$ be two projective homogeneous varieties for semi-simple inner groups $G$ and $G'$ respectively. 
Assuming that $G_{F(X)}$ and $G'_{F(X')}$ are split, the following are equivalent :
\begin{enumerate}
\item[(i)] The motives of $X$ and $X'$ are isomorphic in $\M_F(\mz/p\mz)$;
\item[(ii)] The Poincaré polynomials of $X_{F_{sep}}$ and $X'_{F_{sep}}$ are equal and for all field extensions $L/F$, $G_L$ is split over a prime to $p$ extension of $L$ if and only if $G'_L$ is split over a (possibly different) prime to $p$ extension of $L$. 
\end{enumerate}
\end{cor}
\begin{proof}
For any field extension $L/F$, the function field $L(X_L)$ is purely transcendental if and only if $X_L$ has a rational point. 
Therefore, since $G_{F(X)}$ is split, $G_L$ is split if $X_L$ has a rational point. 
It follows that $G$ is split by a prime to $p$ extension of $L$ if and only if $X_L$ is isotropic mod $p$, or equivalently $M(X_L)$ has non trivial Tate trace. Therefore condition $(ii)$ holds if and only if the Poincaré polynomials of $X_{F_{sep}}$ and $X'_{F_{sep}}$ are equal and 
for all field extensions $L/F$, the Tate traces of $M(X_L)$ and $M(X'_L)$ are either both trivial of both non-trivial.
The implication      
$(i)\Rightarrow (ii)$ follows immediately: if $M(X)$ and $M(X')$ are isomorphic, they have the same Poincaré polynomials over $F_{sep}$ and isomorphic Tate traces over all field extensions $L/F$. \\
Let us now prove the converse. As $G$ is split over the function field $F(X)$, the variety $X$ is maximal among projective homogeneous $G$-varieties with respect to dominance mod $p$. The main result of \cite{K-inner} thus implies that a complete motivic decomposition of $M(X)$ is given by shifts of the upper motive $U_X$. This holds for $X_L$, for any field extension $L/F$. In particular, depending on $U_{X_L}$ being Tate or not, the motive $M(X_L)$ is either anisotropic or pure Tate. In the latter case, $M(X_L)$ is determined by the Poincaré polynomial $P(X_{F_{sep}},t)$. The same holds for $X'$, and as the Poincaré polynomial of $X_{F_{sep}}$ and $X'_{F_{sep}}$ are equal, and $M(X_L)$ and $M(X'_L)$ are pure Tate motives over the same fields $L/F$, it remains to cast Theorem \ref{main.thm}.
\end{proof}

\begin{rem}
If $G$, $G'$ are inner forms of the same quasi-split group and $X$ and $X'$ are of the same type, we may drop the assumption on the Poincaré polynomials of $X_{F_{sep}}$ and $X'_{F_{sep}}$. This applies for instance to Severi-Brauer varieties of the same dimension, or varieties of Borel subgroups for two semi-simple groups of the same type and rank.
\end{rem}

\begin{ex} Let $G$ be a semi-simple group of type $G_2$ (for instance, the automorphism group of a Cayley algebra \cite[§31]{KMRT}) and set $p=2$. The projective homogeneous varieties $X_{\{1\},G}$ and $X_{\{2\},G}$ both have Poincaré polynomial $\sum_{i=0}^5t^i$ over $F_{sep}$. 
%Furthermore, both are generically split by \cite[Proposition 9]{dCG} and are isotropic over a field extension $E/F$ if and only if $G_E$ is split by an odd degree field extension of $E$.  
Furthermore, since a group of type $G_2$ is either split or anisotropic, see~\cite{T}, $G$ is split over $F(X_{\{i\},G})$ for $i=1,2$. 
By Corollary \ref{critborel}, the motives $M(X_{\{1\},G})$ and $M(X_{\{2\},G})$ are thus isomorphic in $\M_F(\mz/2\mz)$. This was originally noticed by Bonnet~\cite{Bo}.
\end{ex}

Assorted with the description of the Tits $p$-indexes of semi-simple groups, Corollary \ref{critborel} produces numerous motivic isomorphisms through cohomological invariants of semi-simple groups (see \cite[Def.\,4, \S IV.3]{dCG} for the definitions of the invariants $b(.)$ and $f_3(.)$).

\begin{cor}\label{cohinv}
Let $p$ be a prime. Assume that $G$ and $G'$ are of the same type as listed below and that $X$ and $X'$ are anisotropic projective $G$-homogeneous and $G'$-homogeneous varieties, also of the same type. Considering motives with coefficients in $\mz/p\mz$, we get
\begin{itemize}
\item[-] ($F_4$ and $p=3$) or ($E_8$ and $p=5$): $M(X)\simeq M(X')$ if and only if $b(G)$ and $b(G')$ generate the same subgroup of $H^3(F,\mathbb{Z}/p\mathbb{Z}(2))$;
\item[-] ($G_2$ and $p=2$): $M(X)\simeq M(X')$ if and only if $G$ and $G'$ are isomorphic;
\item[-] ($^1\!E_6$ and $p=2$): $M(X)\simeq M(X')$  if and only if $f_3(G)=f_3(G')$;
\item[-] ($E_7$ and $p=3$): $M(X)\simeq M(X')$ if and only if \\$b(G)=\pm b(G')\in H^3(F,\mathbb{Z}/3\mathbb{Z}(2))$. 
\end{itemize}
\end{cor}

\begin{proof}
As $X$, $X'$ and $G$, $G'$ are of the same type, the Poincaré polynomials of $X_{F_{sep}}$ and $X'_{F_{sep}}$ are equal. Furthermore the analysis of the Tits $p$-indexes of \cite{dCG} shows that in all the case considered, $G_{F(X)}$ and $G'_{F(X')}$ are split. Corollary \ref{critborel} thus states that $M(X)$ and $M(X')$ are isomorphic if and only if for any field extension $E/F$, $G_E$ is split by a prime-to-$p$ field extension if and only if $G'_E$ is split by a (possibly different) prime-to-$p$ extension. The cohomological characterizations for this property are given in \cite{dCG}.
\end{proof}

\subsection{Motivic equivalence of semi-simple algebraic groups}

The Tits index of a semi-simple algebraic group $G$ over $F$ is a fundamental invariant, which describes the isotropy of $G$. For classical groups, it encompasses classical invariants of algebraic structures over fields, such as the Schur index of central simple algebras or the Witt index of quadratic forms.
It has been known for a while that the isomorphism classes of motives of $G$-projective homogeneous varieties determine the Tits-index of $G$, up to prime-to-$p$ base change. The main result of~\cite{dC} provides a converse. Theorem~\ref{motequivpart} below generalizes this result. The proof, based on Theorem~\ref{main.thm}, is independant of~\cite{dC}. 

We use the same notations as in~\cite[\S I.1]{dCG}. In particular, $\Delta(G)$ denotes the Dynkin diagram of $G$ as well as its set of vertices. The distinguished vertices form a subset denoted by $\delta_0(G)$. If $G$ is of inner type, $\delta_0(G)$ consists of the vertices $\alpha\in\Delta(G)$ such that $X_{\{\alpha\}}$ has a rational point. With our convention for projective homogeneous varieties, $\delta_0(G)$ is empty when $G$ is anisotropic, and coincides with $\Delta(G)$ when $G$ is quasi-split. 

\begin{defi}
Let $G$ and $G'$ be semi-simple algebraic groups over a field $F$, which are inner forms of the same split group. We say that $G$ and $G'$ are motivic equivalent mod $p$ if there is an isomorphism of diagrams$$\varphi:\Delta(G)\longrightarrow \Delta(G'),$$ such that for any subset $\Theta$ of $\Delta(G)$, the motives $M(X_{\Theta,G})$ and $M(X_{\varphi(\Theta),G'})$ are isomorphic with coefficients in $\mz/p\mz$.
\end{defi}

\begin{ex} Consider two central simple algebras $A$ and $A'$ of the same dimension. The groups ${\mathrm {SL}}_1(A)$ and ${\mathrm{SL}}_1(A')$ are motivic equivalent mod $p$ if and only if for any sequence of integers $d_1 < ... < d_k$ the motives of the varieties of flags of right ideals in $A$ and $A'$ of reduced dimensions $d_1,..,d_k$ are isomorphic in $\M_F(\mz/p\mz)$, or the same condition holds after replacing $A'$ by its opposite algebra $A'^{\mathrm{op}}$. A similar statement holds for special orthogonal groups associated to quadratic forms of the same dimension and trivial discriminant, replacing flags of ideals by flags of isotropic subspaces, and adjusting the labelling of the two components of the maximal orthogonal grassmanians if necessary.
\end{ex}

Motivic equivalence of semi-simple algebraic groups aims at giving a classification of these with respect to the motives of their respective projective homogeneous varieties. In view of Corollary \ref{c1}, these motives are determined up to isomorphism by their higher Tate trace. Isotropy of semi-simple groups is controlled by the Tits indexes but in order to prove results in $\M_F(\mz/p\mz)$, one should rather consider their values over $p$-special fields, the Tits $p$-indexes.

A field is $p$-special if any of its finite extensions has degree a power of $p$. For any field $F$ there is a $p$-special closure of $F$, that is a minimal $p$-special field extension $F_p/F$. The $p$-special closures of $F$ are algebraic over $F$ and isomorphic \cite[§101.B]{EKM}. The Tits $p$-index of a semi-simple algebraic group $G$ over $F$ is defined as the Tits index of $G_{F_p}$, where $F_p$ is a $p$-special closure of $F$. The admissible values of the Tits $p$-indexes over fields are determined in \cite{dCG}.

\begin{lem}\label{pspecial}
Let $M\in \IM{F}{\mz/p\mz}$ and $F_p/F$ be a $p$-special closure of $F$. The Tate trace of $M$ is isomorphic over $F_p$ to the Tate trace of $M_{F_p}$.
\end{lem}
\begin{proof}
As $M$ decomposes as a sum of initial motives, it is enough to check that an indecomposable summand $U_X\{k\}$ which is not a Tate motive remains anisotropic over $F_p$. Let $X$ be a variety and assume $X_{F_p}$ is isotropic mod $p$, that is admits a $0$-cycle of prime-to-$p$ degree. Since $F_p$ is a $p$-special field, $X_{F_p}$ actually has a rational point. We claim this implies $X$ is isotropic mod $p$, which concludes the proof by Lemma~\ref{isotropy}. 

To prove the claim, assume that $X_{F_p}$ has a rational point $x$. Replacing $X_{F_p}$ by an affine open subscheme containing $x$, we may assume that $X_{F_p}=\Spec(A_{F_p})$ for some finitely generated $F$-algebra $A$ so that $x$ corresponds to a morphism $A_{F_p}\longrightarrow F_p$ of $F$-algebras. The image of $A$ is a finite $F$-algebra contained in $F_p$, hence a finite field extension $L$ of $F$. The variety $X_L$ has a rational point, and by definition of $p$-special closure, $L/F$ is of prime-to-$p$ degree, giving rise to the needed $0$-cycle on $X$. The converse is obvious. 
\end{proof}
We now prove the main result of this section: 
\begin{thm}\label{motequivpart}
    Let $G$ and $G'$ be semi-simple algebraic groups over a field $F$, which are inner twisted forms of the same split group. Let $\Theta_0$ be a subset of $\Delta(G)$, and consider an isomorphism of Dynkin diagrams
    $$\varphi:\Delta(G)\longrightarrow \Delta(G').$$
The following are equivalent:\\
    
     \begin{enumerate}
        \item[(i)] For any projective homogeneous variety $X_{\Theta,G}$ with $\Theta\supset\Theta_0$, the motives $M(X_{\Theta,G})$ and $M(X_{\varphi(\Theta),G'})$ are isomorphic in $\M_F(\mz/p\mz)$.\\
        \item[(ii)] For any $p$-special field extension $L/F$, $\Theta_0$ is distinguished in $\Delta(G_L)$ if and only if $\varphi(\Theta_0)$ is distinguished in $\Delta(G'_L)$. Moreover, when this holds, $\varphi$ induces a bijection between $\delta_0(G_L)$ and $\delta_0(G'_L)$. 
    \end{enumerate}
\end{thm}

\begin{proof}
$(i)\Rightarrow (ii)$ By assumption the motives $M(X_{\Theta_0,G})$ and $M(X_{\varphi(\Theta_0),G'})$ are isomorphic, hence so are  $M(X_{\Theta_0,G_L})$ and $M(X_{\varphi(\Theta_0),G'_L})$ for all field extensions $L/F$. 
It follows their upper motives are isomorphic. For any field extension $L/F$, $X_{\Theta_0,G_L}$ thus has a $0$-cycle of prime-to-$p$ degree if and only if $X_{\varphi(\Theta_0),G'_L}$ as well. For $p$-special $L/F$, this corresponds to these varieties having a rational point and $\Theta_0$ is distinguished for $\Delta(G_L)$ if and only if $\varphi(\Theta_0)$ is for $\Delta(G'_L)$.

Assume now that $L/F$ is a $p$-special field and $\Theta_0$ is distinguished in $\Delta(G_L)$, which means $\delta_0(G_L)$ contains $\Theta_0$. With the same reasoning as above, assumption $(i)$ with $\Theta=\delta_0(G_L)$ implies that $\varphi(\delta_0(G_L))$ is distinguished in $\Delta(G'_L)$, therefore $\varphi(\delta_0(G_L))\subset \delta_0(G'_L)$. Finally the same proof with $\varphi^{-1}$ and $\varphi(\delta_0(G_L))$ shows that $\varphi^{-1}(\delta_0(G'_L))\subset \delta_0(G_L)$, whence $\delta_0(G'_L)=\varphi(\delta_0(G_L))$. The isomorphism $\varphi$ then identifies the Tits indexes of $G_L$ and $G'_L$.\\

\noindent $(ii)\Rightarrow (i)$ Let $\Theta$ be a subset of $\Delta(G)$ containing $\Theta_0$. We now show that the higher Tate traces of $X_{\Theta,G}$ and $X_{\varphi(\Theta),G'}$ are isomorphic. By Lemma \ref{pspecial}, we only need to work over $p$-special fields containing $F$.\\

Let $L/F$ be a $p$-special field. Since $\Theta$ contains $\Theta_0$, if $\Theta$ is distinguished in $\Delta(G_L)$, then $\Theta_0$ also is and the second part of condition $(ii)$ holds. Therefore, $\Theta$ is contained in $\delta_0(G_L)$ if and only if $\varphi(\Theta)$ is contained in $\delta_0(G'_L)$, that is $X_{\Theta, G_L}$ has a rational point if and only if $X_{\varphi(\Theta),G'_L}$ also does. Since $L$ is $p$-special, having a rational point is equivalent to being isotropic mod $p$ and we get that the varieties $X_{\Theta,G_L}$ and $X_{\varphi(\Theta),G'_L}$ have non-trivial Tate traces over the same $p$-special fields $L/F$. \\

We now proceed by induction on the common rank $n$ of $G$ and $G'$ (that is, the number of vertices of their Dynkin diagrams). 
If 
%$G$ and $G'$ are of rank $1$
$n=1$, then, up to a central isogeny, $G$ and $G'$ are respectively isomorphic ${\mathrm {SL}}_1(Q)$ and ${\mathrm{SL}}_1(Q')$ for some $F$-quaternion algebras $Q$ and $Q'$, and we only have to show that under condition $(ii)$, the motives of the Severi-Brauer varieties of $Q$ and $Q'$ are isomorphic. If $p$ is odd, the motives of $\SB(Q)$ and $\SB(Q')$ are split and isomorphic to the motive of ${\mathbb P}^1_F$. For $p=2$, condition $(ii)$ implies that $M(\SB(Q))$ and $M(\SB(Q'))$ have the same Tate trace over any $2$-special field, hence the same higher Tate trace by Lemma~\ref{pspecial}, and Theorem~\ref{main.thm} finishes the proof.\\

%Now fix semi-simple algebraic groups $G$ and $G'$ of rank greater than $1$ and a distinguished subset $\Theta_0$ of $\Delta(G)$. 
Assume now that 
%the common rank of $G$ and $G'$ is at least $1$, 
$n\geq 2$ and consider a subset $\Theta$ of $\Delta(G)$ such that $\Theta\supset \Theta_0$.
If $\Theta$, hence $\Theta_0$, are empty, then $X_{\Theta,G}$ and $X_{\varphi(\Theta),G'}$ are both isomorphic to $\Spec F$, and their motives are isomorphic. 
If $\Theta$ coincides with $\Delta(G)$, that is $X_{\Theta,G}$ and $X_{\varphi(\Theta),G'}$ are the respective Borel varieties of $G$ and $G'$, condition $(ii)$ implies that $G$ and $G'$ are split over the same $p$-special fields. Hence they satisfy condition $(ii)$ of Corollary~\ref{critborel}, and again they have isomorphic motives. 
Therefore, we may assume $\Theta$ and $\Delta(G)\backslash \Theta$ are not empty. Let $L/F$ be a $p$-special field extension such that  $X_{\Theta,G_L}$ has a rational point. From Chernousov, Gille, Merkurjev motivic decompositions of isotropic projective homogeneous varieties (see \cite{CGM}, \cite[Theorem 7.4]{Bros}), there is an isomorphism $$M(X_{\Theta,G_L})\simeq \bigoplus_{\delta}M(Z_{\delta,G_{L,\Theta}})\{l(\delta)\}$$ where $G_{L,\Theta}$ is the semi-simple part of a Levi component of $G_L$ associated to $\Theta$, the $\delta$'s are the minimal length cosets representatives for some double cosets of the Weyl group of $G$, and $l(.)$ is the length. This data is uniquely determined by the combinatorics of $G$ and does not depend on the choice of a set of simple roots of $G$.\\

The semi-simple algebraic groups $G_{L,\Theta}$ and $G'_{L,\varphi(\Theta)}$ are inner forms of the same split group, 
with rank $r$ satisfying $1\leq r\leq n-1$ : their Dynkin diagrams are obtained from the ones of $G$ and $G'$ by removing the subsets $\Theta$, $\varphi(\Theta)$, respectively. By assumption $(ii)$, for any $p$-special field extension $K/L$, $\varphi$ thus identifies the Tits indexes of $G_K$ and $G'_K$. Induction hypothesis with $\Theta_0=\emptyset$ then gives that $G_{L,\Theta}$ and $G'_{L,\Theta}$ are motivic equivalent mod $p$, so that
$$M(X_{\Theta,G_L})\simeq \bigoplus_{\delta}M(Z_{\delta,G_{L,\Theta}})\{l(\delta)\}\simeq \bigoplus_{\delta}M(Z_{\delta,G'_{L,\varphi(\Theta)}})\{l(\delta)\}\simeq M(X_{\varphi(\Theta),G'_L}).$$
It follows from this isomorphism that $X_{\Theta,G_L}$ and $X_{\varphi(\Theta),G'_L}$ have isomorphic Tate traces for all $p$-special fields $L/F$. Casting Theorem \ref{main.thm}, the motives of $X_{\Theta,G}$ and $X_{\varphi(\Theta),G'}$ are isomorphic as well.
\end{proof}
In the special case $\Theta_0=\emptyset$,  we get the following statement, which is a reformulation of~\cite[Thm. 16]{dC}. 
\begin{cor}
Two semi-simple algebraic groups $G$ and $G'$, inner forms of the same split group, are motivic equivalent mod $p$ relatively to an isomorphism of $\varphi:\Delta(G)\longrightarrow \Delta(G')$ if and only if $\varphi$ induces an isomorphism between the Tits indexes of $G_L$ and $G'_L$, for any $p$-special field extension $L/F$. 
\end{cor}

\section{Motivic splitting patterns}

Recall that the splitting pattern of a quadratic form $q$ over $F$ is defined as the set of values taken by the Witt index of $q$ over all field extensions, see~\cite{HR}, ~\cite[\S 25]{EKM}. It is thus a finite set of integers, bounded by half the dimension of $q$. As shown by Knebusch~\cite{Kne1}, the splitting pattern of $q$ is obtained considering a \emph{generic splitting tower} of $q$. More precisely, the splitting pattern is equal to the set of Witt indices $\{i_w(q_{F_0}),\dots, i_w(q_{F_h})\}$, where $F_0=F$, for $i\geq 1$, $F_i$ is the function field of the anisotropic part of $q_{F_{i-1}}$ and $h$ is the height of $q$, that is the smallest integer such that $q_{F_{h}}$ has anisotropic part of dimension at most $1$.

We define a motivic version of the splitting pattern for projective homogeneous varieties and their direct summands as follows: 
\begin{defi}
\label{splittingpattern.def}
Let $X$ be a projective homogeneous variety for some $p$-inner semi-simple group $G$ and let $M$ be a direct summand of $M(X)$ in $\M_F(\mz/p\mz)$. We say that a pure Tate motive $T$ belongs to the splitting pattern of $M$ if there is a field extension $K/F$ such that $T_K$ is isomorphic to the Tate trace of $M_K$. The set of such pure Tate motives is the motivic splitting pattern of $M$.
\end{defi}

\begin{ex}
Let $Q$ be a projective quadric defined by a $2m$ dimensional quadratic form $q$. In view of Example~\ref{qf.ex}, the motivic splitting pattern of $Q$ with coefficients in $\mz/2\mz$ consists of the pure Tate motives 
\[\sum_{k=0}^{i-1}\,\biggl(\mz/2\mz\{k\}+\mz/2\mz\{2m-2-k\}\biggr),\]

where $i$ runs through the (classical) splitting pattern of the quadratic form $q$. 
\end{ex}

The motivic splitting pattern of $M$ may be described using partial splitting fields, as we now proceed to show. 
Let $M$ be a direct summand of the motive of a projective homogeneous variety over $F$ for a $p$-inner semi-simple group $G$, with coefficients in $\mz/p\mz$.
Fix a splitting family $\mathcal{C}_M=\{F(X^0_1),...,F(X^0_{r_0})\}$ of $M$, where the varieties $X^0_k$ are quasi-$G$-homogeneous varieties (see Example \ref{phv.ex}).

If $M$ is not a pure Tate motive, there exists an integer $k_0$ such that $X^0_{k_0}$ is anisotropic mod $p$. Let $F_0=F(X^0_{k_0})$. 
The Tate trace of the motive $M_{F_0}$ has greater rank than the Tate trace of $M$. Consider a splitting family $\mathcal{C}_{M_{F_0}}=\{F_0(X^1_1),...,F_0(X^1_{r_1})\}$ of $M_{F_0}$. If $M_{F_0}$ is not pure Tate, we pick again an anisotropic $X^1_{k_1}$ and let $F_1=F_0(X^1_{k_1})$. After a finite number of such field extensions, the motive $M$ becomes a pure Tate motive, as the rank of the Tate trace increases at each step. 
The corresponding sequence of field extensions $$F \hookrightarrow F_0 \hookrightarrow F_1 \hookrightarrow ... \hookrightarrow F_s$$
is called a splitting tower of $M$. 

\begin{rem}
Let $Q$ be a projective quadric, defined by a quadratic form $q$. Assume $q$ is non-split, and denote by $Q_{\mathrm{an}}$ the quadric defined by its anisotropic part. By~\cite{Rost}, see also Example~\ref{qf.ex}, the motive $M(Q)$ with coefficients in $\mz/2\mz$ contains a shift of $M(Q_{\mathrm{an}})$, hence of its upper motive $U_{Q_{\mathrm{an}}}$, as a direct summand. Therefore, we may take $X^0_{k_0}=Q_{\mathrm{an}}$ in the process above. Iterating the process, this proves that the generic splitting tower of $q$ introduced by Knebusch in~\cite{Kne1} is a splitting tower for the motive $M(Q)$.
\end{rem}

\begin{prop}
Let $G$ be a semi-simple group of inner type and $X$ a projective $G$-homogeneous variety. Consider a direct summand $M$ of the motive $M(X)$ in $\M_{F}({\mz/p\mz})$. 
A pure Tate motive $T$ belongs to the motivic splitting pattern of $M$ if and only if $T_E$ is isomorphic to the Tate trace of $M_E$, for some field extension $E/F$ appearing in a splitting tower of $M$.
\end{prop}

\begin{proof}
We proceed by induction on the anisotropic rank of $M$, that is the integer $\rk(M)-\rk\bigr(\Tr(M)\bigl)$. If this is $0$, $M$ is a pure Tate motive, $\Tr(M)$ is the only element in the motivic splitting pattern of $M$ and we may take $E=F$. Assume now that $\rk(M)>\rk\bigl(\Tr(M)\bigr)$ and consider a pure Tate motive $T$ which is in the motivic splitting pattern of $M$. 
If $T=\Tr(M)$, again we may take $E=F$, so we assume $T$ has rank strictly larger than $\Tr(M)$. By definition of the motivic splitting pattern, $\Tr(M_K)$ is isomorphic to $T_K$ for some field extension $K/F$. Consider a splitting family of $M$,  $\mathcal{C}_M=\{F(X_1),...,F(X_{n})\}$, consisting of function fields of projective homogeneous varieties given by a decomposition of $M$ as a direct sum of indecomposable summands 
$M=\sum_{k=1}^{n}U_{X_k}\{i_k\}$. Since $\Tr(M_K)$ has rank strictly larger than $\Tr(M)$, there is an integer $1\leq r \leq n$ such that $X_r$ is anisotropic mod $p$ and becomes isotropic mod $p$ over $K$. Consider the following diagram of fields 
$$
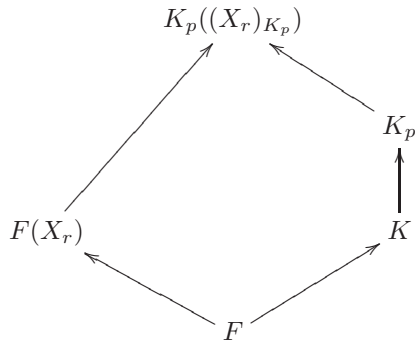
\xymatrix{ &K_p((X_r)_{K_p})&\\ &&K_p \ar[ul]
\\F(X_r)\ar[uur] &&  K \ar[u]
\\ & F \ar[ul] \ar[ur] }$$
where $K_p$ denotes the $p$-special closure of $K$.

Since $X_r$ is isotropic mod $p$ over $K$, it has a rational point over its $p$-special closure $K_p$ and the field extension $K_p((X_r)_{K_p})/K_p$ is purely transcendental \cite[Thm. 3.10]{KeRe}. Therefore, the Tate trace of $M$ is preserved under this extension, and combining with Lemma~\ref{pspecial}, we get that 
$\Tr(M_{K_p((X_r)_{K_p})})=T_{K_p((X_r)_{K_p})}$. In particular, $T$ belongs to the splitting pattern of $M_{F(X_r)}$. We get by induction a field extension $E/F(X_r)$ appearing in a splitting tower of $M_{F(X_r)}$ such that $T_E$ is isomorphic to the Tate trace of $M_E$. Completing this tower of extensions with $F(X_r)/F$, $E$ belongs to a splitting tower of $M$ as well.
\end{proof}


\begin{thebibliography}{99}

\bibitem{amitsur}
S. A. Amitsur, Generic Splitting Fields of Central Simple Algebras, Annals of Mathematics, Second Series, volume 62, no. 1, 8--43, 1955. 

\bibitem{Bass}
H. Bass, Algebraic K-theory. W. A. Benjamin, Inc., New York-Amsterdam, 1968.

\bibitem{bachmann}
T. Bachmann, On the invertibility of motives of affine quadrics, Doc. Math. 22, 363-395, 2017.

\bibitem{bachmannvishik}
T. Bachmann and A. Vishik, Motivic equivalence of affine quadrics, Math. Annalen, 371, No.1, 741-751, 2018.

\bibitem{Bo}
J.-P. Bonnet, A motivic isomorphism between two projective homogeneous varieties under the action of a group of type $G_2$ (French). Doc. Math. 8, 247--277 (2003). 

\bibitem{Bros}
P. Brosnan, On motivic decompositions arising from the method of Białynicki-Birula, Invent. Math. 161, No. 1, 91--111, 2005.

\bibitem{CGM} 
V. Chernousov, S. Gille and A. Merkurjev, Motivic decomposition of isotropic projective homogeneous varieties. Duke Math. J. 126, No. 1, 137-159 (2005). 

\bibitem{CM} 
V. Chernousov and A. Merkurjev, Motivic decompositions of projective homegeneous varieties and the Krull-Schmidt theorem. 
Transformation Groups 11, no. 3, 371-386, 2006. 

\bibitem{dC}
C. De Clercq, \'Equivalence motivique des groupes semi-simples, Compositio Mathematica 153, Issue 10, 2195-2213, 2017.

\bibitem{dCdec}
C. De Clercq, Motivic decompositions of projective homogeneous varieties and change of coefficients, C. R., Math., Acad. Sci. Paris 348, No. 17-18, 955-958, 2010.

\bibitem{dCG}
C. De Clercq and S. Garibaldi, Tits p-indexes of semisimple algebraic groups, Journal of the London Mathematical Society, 95, Issue 2, 567-585, 2017.

\bibitem{dCQZ}
C. De Clercq, A. Quéguiner and M. Zhykhovich, Critical varieties and motivic equivalence for algebras with involution, Trans. Am. Math. Soc. 375, No. 11, 7529--7552, 2022.

\bibitem{EKM}
R. Elman, N.  Karpenko and A. Merkurjev, The algebraic and geometric theory of quadratic forms, American Mathematical Society Colloquium Publications {56}, 2008. 	

\bibitem{GPS}
S. Garibaldi, V. Petrov, N. Semenov, Shells of twisted flag varieties and the Rost invariant, Duke Math. J. 165, no. 2, 285-339, 2016.

\bibitem{GS} 
P. Gille and T. Szamuely, Central simple algebras and Galois cohomology, Cambridge studies in advanced mathematics {101}, 2006.

\bibitem{HR}
Hurrelbrink J. and Rehmann U, Splitting Patterns of Quadratic Forms. Math.Nachr. 176, 111-127, 1995. 

\bibitem{Kahn-Bourbaki}
B. Kahn, Formes quadratiques et cycles alg\'ebriques, dans S\'eminaire Bourbaki : volume 2004/2005, expos\'es 938-951, Ast\'erisque, no. 307, Expos\'e no. 941, pp. 113-163, 2006.

\bibitem{K-anisotropy} 
N. Karpenko, On anisotropy of orthogonal involutions. J. Ramanujan Math. Soc. 13, No. 1, 1-22, 2000. 

\bibitem{K:motivic}
N. Karpenko, Criteria of motivic equivalence for quadratic forms and central simple algebras, Math. Ann. 317, 585-611, 2000.

\bibitem{K-motif}
N. Karpenko, Cohomology of relative cellular spaces and isotropic flag varieties, St. Petersburg Math. J. 12, no. 1, 1-50, 2001.

\bibitem{K-ICM}
N. Karpenko, Canonical dimension, Proceedings of the ICM 2010, vol. II, 146--161.

\bibitem{K-inner} 
N. Karpenko, Upper motives of algebraic groups and incompressibility of Severi-Brauer varieties, J. Reine Angew. Math. 677, 179--198, 2013.

\bibitem{K-outer} 
N. Karpenko, Upper motives of outer algebraic groups. In Quadratic forms, linear algebraic
groups, and cohomology, vol. 18 of Dev. Math. Springer, New York, 249--258, 2010.

 \bibitem{K-Hoff} 
N. Karpenko, On the first Witt index of quadratic forms.
Invent. Math. 153, no. 2, 455-462, 2003.

\bibitem{Kar-Mer1}
N. Karpenko and  A. Merkurjev, Essential dimension of quadrics, Invent. Math. 153, no. 2, 361-372, 2003.

\bibitem{KZ}
N. Karpenko and M. Zhykhovich, Isotropy of unitary involutions. Acta Math. 211, no. 2, 227-253, 2013. 

\bibitem{KeRe}
I. Kersten, U. Rehmann, Generic splitting of reductive groups, Tohoku Mathematical Journal 46, 35-70, 1994.

\bibitem{Kne1}
M. Knebusch, Generic splitting of quadratic forms, I, Proc. London Math. Soc. 33, 65-93, 1976.

\bibitem{KMRT}
M. Knus, A. Merkurjev, M. Rost and J.-P. Tignol., {The book of involutions}, Amer. Math. Soc., Providence, RI, 1998.

\bibitem{Kock} 
B. Köck, Chow motive and higher Chow theory of $G/P$. Manuscrip. Math. 70, no. 4, 363-372 (1991). 

\bibitem{NSZ}
S. Nikolenko, N. Semenov, K. Zainoulline, Motivic decomposition of anisotropic varieties of type $F_4$ into generalized Rost motives, J. of K-theory 3, no. 1, 85-102, 2009.
  
\bibitem{PS}
V. Petrov, N. Semenov, Generically split projective homogeneous varieties, Duke Math. J. 152, No. 1, 155-173, 2010.
  
\bibitem{PSZ}
V. Petrov, N. Semenov, K. Zainoulline, J-invariant of linear algebraic groups, Ann. Sci. \'Ec. Norm. Sup. 41, 1023-1053, 2008.  
  
\bibitem{Rost} 
M. Rost, The motive of a Pfister form. Preprint. 

\bibitem{Srini}
S. Srinivasan, Motivic decomposition of projective pseudo-homogeneous varieties, Transform. Groups 22, No. 4, 1125-1142 2017; correction ibid. 24, No. 4, 1309-1311, 2019.

\bibitem{T}
J. Tits, Classification of algebraic semi-simple groups, Algebraic Groups and Discontinuous Subgroups, Proc. Symp. Pure Math., 9, 33-62, Amer. Math. Soc., 1966.

\bibitem{vish-lens}
A. Vishik, Motives of quadrics with applications to the theory of quadratic forms. Lecture Notes in Math. 1835, Proceedings of the Summer School ``Geometric Methods in the Algebraic Theory of Quadratic Forms'' Lens 2000, 25--101, 2004.

\bibitem{vishik:isotropic}
A. Vishik, Isotropic motives, J. Inst. Math. Jussieu 21, No. 4, 1271-1330, 2022.

\bibitem{vish-uinv}
A. Vishik, Fields of u-invariant $2^r+1$, Algebra, Arithmetic and Geometry - In Honor of Yu.I.Manin", Birkhauser, 661-685, 2010.

\bibitem{voe}
V. Voevodsky, Motivic cohomology with $\mz/2$ coefficients, Publ. Math., Inst. Hautes Étud. Sci. 98, 59--104, 2003.
\end{thebibliography}
\end{document}